\documentclass[12pt]{amsart}

\usepackage{hyperref}
\usepackage[latin1]{inputenc}

\usepackage{mfdmetrics}

\usepackage[margin=1in]{geometry}

\author{Brian Clarke}
\date{\today}
\title[The Metric Geometry of the Manifold of Riemannian Metrics]{The
  Metric Geometry of the Manifold of Riemannian Metrics over a Closed
  Manifold}
\thanks{This research was supported by the International Max Planck
  Research School \emph{Mathematics in the Sciences} and the DFG
  Research Training Group \emph{Analysis, Geometry and their
    Interaction with the Natural Sciences}.}
\address{Department of Mathematics, Stanford University, Stanford, CA
  94304-2125}
\email{\href{mailto:bfclarke@math.stanford.edu}{bfclarke@math.stanford.edu}}
\urladdr{\href{http://math.stanford.edu/~bfclarke/}{http://math.stanford.edu/${\sim}$bfclarke/}}

\begin{document}

\begin{abstract}
  We prove that the $L^2$ Riemannian metric on the manifold of all
  smooth Riemannian metrics on a fixed closed, finite-dimensional
  manifold induces a metric space structure.  As the $L^2$ metric is a
  weak Riemannian metric, this fact does not follow from general
  results.  In addition, we prove several results on the exponential
  mapping and distance function of a weak Riemannian metric on a
  Hilbert/Fréchet manifold.  The statements are analogous to, but
  weaker than, what is known in the case of a Riemannian metric on a
  finite-dimensional manifold or a strong Riemannian metric on a
  Hilbert manifold.
\end{abstract}

\maketitle

\section{Introduction}\label{sec:introduction}

This paper is the first in a pair studying the metric geometry of the
Fréchet manifold $\M$ of all $C^\infty$ Riemannian metrics on a fixed
closed, finite-dimensional, orientable manifold $M$.  This manifold
has a natural weak Riemannian metric called the \emph{$L^2$ metric},
which will be defined in Section \ref{sec:preliminaries}.  The main
result of the paper is the following.

\begin{thm*}
  $(\M, d)$, where $d$ is the distance function induced from the $L^2$
  metric, is a metric space.
\end{thm*}

This is indeed a nontrivial theorem, as the fact that the $L^2$ metric
is a weak (as opposed to strong) Riemannian metric implies that its
induced distance function is \emph{a priori} only a pseudometric.
That is, some points may have distance zero from one another.  The
first authors we know to have recognized this are Michor and Mumford.
In
\cite{michor05:_vanis_geodes_distan_spaces_of,michor06:_rieman_geomet_spaces_of_plane_curves},
they found examples of weak Riemannian metrics on Fréchet manifolds of
embeddings for which the distance between \emph{any} two points is zero.
They then constructed other weak Riemannian metrics that they proved
induce metric space structures on the manifolds (i.e., have
positive-definite distance functions).

The manifold of metrics and the $L^2$ metric have been of interest to
mathematicians and physicists for some time.  We first became
interested in their study because of their applications to Teichmüller
theory, developed by Fischer and Tromba \cite{tromba-teichmueller}.
If the base manifold $M$ is a Riemann surface of genus greater than
one, then the Teichmüller space of $M$ is diffeomorphic to $\Mhyp /
\DO$, where $\Mhyp \subset \M$ is the submanifold of hyperbolic
metrics (with constant scalar curvature $-1$) and $\DO$ is the group
of diffeomorphisms that are homotopic to the identity, acting by
pull-back.  The $L^2$ metric on $\M$ descends to $\Mhyp / \DO$ and is
isometric (up to a constant scalar factor) to the much-studied
Weil-Petersson metric on Teichmüller space.  The sequel
\cite{clarke:_compl_of_manif_of_rieman_metric} to this paper contains
an application to the geometry of Teichmüller space with respect to a
class of metrics we define on it that generalize the Weil-Petersson
metric.

The first investigations of the geometry of the manifold of metrics,
which essentially avoided the infinite-dimensional issues that arise,
were undertaken by DeWitt as part of his Hamiltonian formulation of
general relativity \cite{dewitt67:_quant_theor_of_gravit}.  Soon
thereafter, Ebin \cite{ebin70:_manif_of_rieman_metric} used the $L^2$
metric to investigate the differential topology of $\M$, as well as
its quotient by the diffeomorphism group.

Later, the basic Riemannian geometry of the $L^2$ metric was
independently investigated by Freed and Groisser
\cite{freed89:_basic_geomet_of_manif_of} as well as Gil-Medrano and
Michor \cite{gil-medrano91:_rieman_manif_of_all_rieman_metric}.  (The
latter paper even considers such questions for open base manifolds
$M$.)  Many of the results of the two papers coincide (both compute
the geodesics and curvature of the $L^2$ metric, for instance), and
the crucial observation of both papers is that the computation of
certain essential geometric quantities is \emph{pointwise} in nature.
For example, a geodesic on $\M$ is a one-parameter family $g_t$ of
metrics on $M$, uniquely determined by an initial point $g_0$ and
initial tangent vector $g_0'$.  However, more is true: for each $x \in
M$, the value of $g_t(x)$ depends only on $g_0(x)$ and $g_0'(x)$, and
not on the values of $g_0$ or $g_0'$ at any other points of $M$.

As will be seen in the subsequent sections,
the main challenge in studying the metric geometry of the
$L^2$ metric is in moving from \emph{pointwise} questions to
\emph{local} or even \emph{global} questions.

The paper is organized as follows.

In Section \ref{sec:preliminaries}, we present the necessary facts and
notation needed for studying the manifold of Riemannian metrics.  In
Section \ref{sec:weak-riem-manif}, we review the definition of a
Riemannian metric on a Hilbert or Fréchet manifold, paying special
attention to the distinction between weak and strong metrics.  We also
prove some results on the exponential mapping and the induced distance
function of a weak Riemannian manifold for which we could not find a
rigorous treatment in the weak case.  In Section
\ref{sec:basic-metr-geom}, we prove the main theorem mentioned above.
The results of Section \ref{sec:weak-riem-manif} turn out to be
inadequate in this situation, but we give a direct proof for the $L^2$
metric.

\subsection*{Acknowledgments}

The results of this paper formed a portion of my Ph.D.~thesis
(\cite{clarked):_compl_of_manif_of_rieman}, where the reader may find
the facts here presented in significantly greater detail), written at
the Max Planck Institute for Mathematics in the Sciences in Leipzig
and submitted to the University of Leipzig.  I would like to thank my
advisor Jürgen Jost for his years of patient assistance.  I am also
grateful to Guy Buss, Christoph Sachse, and Nadine Große for many
fruitful discussions and their careful proofreading.

\section{The Manifold of Riemannian Metrics}\label{sec:preliminaries}

The basic facts about the manifold of Riemannian metrics given in this
section can be found in \cite{freed89:_basic_geomet_of_manif_of},
\cite{gil-medrano91:_rieman_manif_of_all_rieman_metric}, and \cite[\S
2.5]{clarked):_compl_of_manif_of_rieman}

\subsection{The $L^2$ Metric}\label{sec:manif-riem-metr}

For the entirety of the paper, let $M$ denote a fixed closed,
orientable, $n$-dimensional $C^\infty$ manifold.

We denote by $S^2 T^*M$ the vector bundle of symmetric $(0,2)$ tensors
over $M$, and by $\s$ the Fréchet space of $C^\infty$ sections of $S^2
T^* M$.  The space $\M$ of Riemannian metrics on $M$ is an open subset
of $\s$, and hence it is trivially a Fréchet manifold, with tangent
space at each point canonically identified with $\s$.  (For a detailed
treatment of Fréchet manifolds, see, for example,
\cite{hamilton82:_inver_funct_theor_of_nash_and_moser}.  For a more
thorough treatment of the differential topology and geometry of $\M$,
see \cite{ebin70:_manif_of_rieman_metric}.)

$\M$ carries a natural Riemannian metric $(\cdot, \cdot)$, called the
\emph{$L^2$ metric}, induced by integration from the natural scalar product
on $S^2 T^* M$.  Given any $g \in \M$ and $h, k \in \s \cong T_g \M$,
we define
\begin{equation}\label{eq:1}
  (h, k)_g := \integral{M}{}{\tr_g(hk)}{d \mu_g}.
\end{equation}
Here, $\tr_g(hk)$ is given in local coordinates by $\tr(g^{-1} h
g^{-1} k) = g^{ij} h_{il} g^{lm} k_{jm}$, and $\mu_g$ denotes the
volume form induced by $g$.  We will denote the norm induced by
$(\cdot, \cdot)_g$ with $\normdot_g$.

Throughout the paper, we use the notation $d$ for the distance
function induced from $(\cdot, \cdot)$ by taking the infimum of the
lengths of paths between two given points.

The basic Riemannian geometry of $(\M, (\cdot, \cdot))$ is relatively
well understood.  For example, it is known that the sectional
curvature of $\M$ is nonpositive
(cf. \cite[Cor.~1.17]{freed89:_basic_geomet_of_manif_of}), and the
geodesics of $\M$ are known explicitly (cf. Section
\ref{sec:geodesics}).

We will consider related structures restricted to a point $x \in M$.
Let $\satx := S^2 T^*_x M$ denote the vector space of symmetric
$(0,2)$-tensors at $x$, and let $\Matx \subset \satx$ denote the open
subset of tensors inducing a positive definite scalar product on $T_x
M$.  Then $\Matx$ is an open submanifold of $\satx$, and its tangent
space at each point is canonically identified with $\satx$. For each
$g \in \Matx$, we define a scalar product $\langle \cdot , \cdot
\rangle_g$ on $T_g \Matx \cong \satx$ by setting, for all $h, k \in
\satx$,
\begin{equation*}
  \langle h, k \rangle_g := \tr_g(hk).
\end{equation*}
Then $\langle \cdot , \cdot \rangle$ defines a Riemannian metric on
the finite-dimensional manifold $\Matx$.

\subsection{A Product Manifold Structure for
  $\M$}\label{sec:manif-posit-funct}

$\M$ can be written globally as a product manifold, with the factors
given by the manifold of metrics inducing a given volume form and the
manifold of volume forms on $M$.  We sketch the details of this here.

The set of smooth volume forms on $M$, denoted by $\V$, is a Fréchet
manifold.  In fact, it is an open subset of $\Omega^n(M)$, the Fréchet
space of highest-order differential forms on $M$, and therefore its
tangent spaces are canonically identified with $\Omega^n(M)$.

Given any volume form $\mu \in \V$ and any $n$-form $\alpha \in
\Omega^n(M)$, there exists a unique $C^\infty$ function, denoted by
$(\alpha / \mu)$, such that
\begin{equation}\label{eq:118}
  \alpha = \left( \frac{\alpha}{\mu} \right) \mu.
\end{equation}
If $\alpha$ is also a smooth volume form, then $(\alpha / \mu)$ is
additionally a strictly positive function.

Now, for any fixed volume form $\mu \in \V$, let $\Mmu \subset \M$
denote the set of metrics inducing the volume $\mu$.  It is a smooth
submanifold of $\M$ \cite[\S 8]{ebin70:_manif_of_rieman_metric} with
tangent space
\begin{equation*}
  T_g \Mmu = \{ h \in \s \mid \tr_g h = 0 \}.
\end{equation*}
This follows from the fact that the differential of the map $g \mapsto
\mu_g$ at the point $g_0$ is $h \mapsto \frac{1}{2} \tr_{g_0}(h)
\mu_{g_0}$ \cite[Lemma~2.38]{clarked):_compl_of_manif_of_rieman}.

We now define a map
\begin{align*}
  i_\mu : \V \times \Mmu &\rightarrow \M \\
  (\nu, g) &\mapsto \left( \frac{\nu}{\mu} \right)^{2/n} g.
\end{align*}
It is not hard to see that $\mu_{i_\mu (\nu, g)} = \nu$, as well as
that $i_\mu$ is a diffeomorphism.

$\Mmu$ inherits a Riemannian metric as a submanifold of $(\M, (\cdot,
\cdot))$.  We can also pull back the $L^2$ metric on $\M$ to $\V$ via
$i_\mu$, namely by choosing any $g \in \Mmu$ and noting that $i_\mu|
\V \times \{g\}$ is an embedding of $\V$ into $\M$.  A relatively
straightforward computation  then shows that this
pull-back metric is given by
\begin{equation}\label{eq:2}
  (\!( \alpha, \beta )\!)_\nu = \frac{4}{n} \integral{M}{}{\biggl(
      \frac{\alpha}{\nu} \biggr) \left(
      \frac{\beta}{\nu} \right)}{d \nu}, \qquad \alpha, \beta \in
  \Omega^n(M) \cong T_\nu \V.
\end{equation}
Thus, this metric is independent of our choices of $g$ and $\mu$.  In
fact, it is just the constant factor $\frac{4}{n}$ times the most
obvious Riemannian metric on $\V$.

\subsection{Geodesics}\label{sec:geodesics}

As noted above, the geodesic equation of $\M$ can be solved explicitly
(see \cite[Thm.~2.3]{freed89:_basic_geomet_of_manif_of},
\cite[Thm.~3.2]{gil-medrano91:_rieman_manif_of_all_rieman_metric}).
We will not need it in full generality here, but we will need the
geodesics associated to the product manifold structure of $\M$.  They
are given in the following two propositions.

\begin{prop}[{\cite[Prop.~2.1]{freed89:_basic_geomet_of_manif_of}}]\label{prop:5}
  The geodesic in $\V$ starting at
  $\nu_0$ with initial tangent $\alpha$ is given by
  \begin{equation*}
    \nu_t = \left( 1 + \frac{t}{2} \biggl( \frac{\alpha}{\nu_0}
      \biggr) \right)^2 \nu_0.
  \end{equation*}
  As a result, for every $\nu_0 \in \V$, the exponential mapping
  $\exp_{\nu_0}$ is a diffeomorphism from an open set $U \subset T_{\nu_0}
  \V$ onto $\V$.
\end{prop}

\begin{prop}[{\cite[Thm.~8.9]{ebin70:_manif_of_rieman_metric} and
    \cite[Prop.~2.2]{freed89:_basic_geomet_of_manif_of}}]\label{prop:1}
  The submanifold $\M_\mu$ is a globally symmetric space, and the
  geodesic starting at $g_0$ with initial tangent $g'_0 = h$ is given
  by
  \begin{equation*}
    g_t = g_0 \exp (t H),
  \end{equation*}
  where $H := g_0^{-1} h$.

  In particular, $\M_\mu$ is geodesically complete, and $\exp_g$ is a
  diffeomorphism from $T_g \M_\mu$ to $\M_\mu$ for any $g \in \M_\mu$.
\end{prop}

As we will see in Section \ref{sec:basic-metr-geom}, Propositions
\ref{prop:5} and \ref{prop:1} are sufficient to prove that $\V$ and
$\Mmu$, together with the Riemannian metrics induced from the $L^2$
metric on $\M$, are metric spaces.  However, we will also see that to
prove that the $L^2$ metric on $\M$ itself induces a metric space
structure, the geodesic equation is insufficient.

\section{Weak Riemannian manifolds}\label{sec:weak-riem-manif}

A Riemannian metric on a Hilbert or Fréchet manifold is a smooth
$(0,2)$-tensor that induces a bounded, positive-definite scalar
product on each tangent space.  There are two kinds of Riemannian
metrics on infinite-dimensional manifolds: if the tangent spaces are
complete with respect to the scalar product induced by the metric, it
is called \emph{strong}.  Otherwise, it is called \emph{weak}.  As we
will see, weak metrics are significantly more technically challenging
than strong metrics (which are treated in, e.g.,
\cite{klingenberg95:_rieman_geomet} and
\cite{lang95:_differ_and_rieman_manif}).  However, in the case of a
proper Fréchet space (where the topology does not come from any single
norm), only weak Riemannian metrics are possible.  This, among other
considerations like the pointwise nature of certain geometric
quantities mentioned in the introduction, cements their importance in
global analysis and the value of their study.

The subtle but important distinction between weak and strong
Riemannian metrics leads to a vast gulf in their properties.  For a
strong Riemannian metric, one can reproduce many of the important
results from finite-dimensional Riemannian geometry
\cite{klingenberg95:_rieman_geomet,lang95:_differ_and_rieman_manif}.
For example, the Levi-Civita connection, geodesics, and the
exponential mapping exist.  A strong Riemannian metric induces a
distance function that gives a metric space structure on the manifold.
In addition, the metric topology agrees with the manifold topology.

None of the above-mentioned results hold in general for weak
Riemannian manifolds, though some can be directly shown for many
important examples.

In this section, we will give some basic results on the distance
function of a weak Riemannian manifold.  We have not found these
results formally recorded anywhere, though they may be known to
experts in the field.  Our approach essentially follows that of
\cite[\S 1.8]{klingenberg95:_rieman_geomet}, which treats the case of
strong Riemannian manifolds.  We have made the necessary adjustments
to the results and proofs so that they hold in the weak case.

For the remainder of this section, let $(N, \gamma)$ be a Riemannian
Fréchet manifold.  Just as in the case of finite-dimensional
Riemannian manifolds, we can use $\gamma$ to define a distance between
points of $N$ by taking the infima of lengths of paths.  It is then
clear that this distance function is a \emph{pseudometric}, but as in
\cite{michor05:_vanis_geodes_distan_spaces_of,
  michor06:_rieman_geomet_spaces_of_plane_curves}, it may fail to be
positive definite.  The problem in showing positive-definiteness on a
weak Riemannian manifold is that the exponential mapping and its
inverse need not be defined on an open $\gamma$- or $d$-ball,
respectively.  On the other hand, this is a vital ingredient in the
proof for strong Riemannian metrics.

\subsection{The exponential mapping and distance function on a weak
  Riemannian manifold}\label{sec:expon-mapp-dist-1}

We cannot prove much that is useful about weak Riemannian manifolds
without first making a couple of assumptions on the exponential
mapping---this is to avoid the problems encountered in, e.g., \cite[\S
1]{constantin03:_geodes_flow_diffeom_group_of_circl} where the
exponential mapping exists but is not a local diffeomorphism.

\begin{dfn}\label{dfn:20}
  We call a weak Riemannian manifold $(N, \gamma)$ \emph{normalizable
    at $x$} if there are open neighborhoods $U_x \subseteq T_x N$ and
  $V_x \subseteq N$ of $0$ and $x$, respectively, such that
  \begin{enumerate}
  \item $U_x$ is convex and
  \item the exponential mapping $\exp_x$ exists and is a
    $C^1$-diffeomorphism between $U_x$ and $V_x$.
  \end{enumerate}
  Note that the neighborhoods $U_x$ and $V_x$ are required to be open
  in the manifold topology of $N$.  We do not require that $U_x$ be
  open in the topology induced by $\gamma$.

  We call $(N, \gamma)$ \emph{normalizable} if it is normalizable at
  each $x \in N$.
\end{dfn}

As in finite-dimensional Riemannian geometry, for $x \in N$ the length
of a radial geodesic $\alpha(t) := \exp_x (t v)$, with $v \in T_x N$,
is $\| v \|_\gamma$.  The finite-dimensional proof can be carried over
word-for-word.

\begin{dfn}
  Let $x \in N$.  We denote by $S_x N \subset T_x N$ the unit sphere,
  i.e.,
  \begin{equation*}
    S_x N = \{ v \in T_x N \mid \| v \|_\gamma = 1 \}.
  \end{equation*}
\end{dfn}

For the rest of this section, let $(N, \gamma)$ be a weak Riemannian
manifold that is normalizable at a point $x \in N$, and retain the
notation of Definition \ref{dfn:20}.

\begin{lem}\label{lem:1}
  The function
  \begin{equation*}
    \begin{aligned}
      R : T_x N \setminus \{0\} &\rightarrow \R_{\geq 0} \\
      v &\mapsto \sup \{ r \in \R_{\geq 0} \mid r \cdot v \in U_x \},
    \end{aligned}
  \end{equation*}
  is continuous.

  Additionally, for each $v \in T_x N \setminus \{0\}$, $R(v) > 0$.
\end{lem}
\begin{proof}
  $R$ is the multiplicative inverse of the Minkowski functional of
  $U_x$, so continuity of $R$ follows from
  \cite[II.1.6]{schaefer99:_topol_vector_spaces} by the convexity and
  openness of $U_x$.

  Now, let $v \in T_x N$ be given.  Since $T_x N$ with its manifold
  topology is a topological vector space and $U_x$ is a neighborhood
  of the origin, there is some $\epsilon > 0$ such that $\epsilon
  \cdot v \in U_x$.  Thus $R(v) > 0$.
\end{proof}

\begin{rmk}
  Lemma \ref{lem:1} does not imply that $R(v)$ is uniformly bounded
  away from zero, even if we restrict the domain of $R$ to $S_x N$ at
  each $x \in N$.
\end{rmk}

The next two propositions allow us to control the lengths of paths
based at $x$, provided they lie within the neighborhood $U_x$.  We
will remark below on how these statements have been weakened from the
strong case.

\begin{prop}\label{prop:4}
  Let $r(s) \cdot v(s) \in U_x$, $s \in [0,1]$, be a path in $U_x$
  such that $v(s) \in S_x N $, $r(s) \in \R_{\geq 0}$.  (That is, we
  express the path in polar coordinates.)  We define a path $\alpha$
  by $\alpha(s) := \exp_x (r(s) v(s))$, $s \in [0,1]$.  Then
  \begin{equation*}
    L(\alpha) \geq | r(1) - r(0) |,
  \end{equation*}
  with equality if and only if $v(s)$ is constant and $r'(s) \geq 0$.
\end{prop}
\begin{proof}
  By Lemma \ref{lem:1}, as well as the compactness of $[0,1]$, there
  exist $\epsilon, \delta > 0$ such that if
  \begin{equation*}
    (s,t) \in U_{\epsilon,\delta} := 
    \left\{
      (s,t) \in \R^2 \mid s \in [0,1],\ t \in [-\epsilon, r(s) + \delta]
    \right\},
  \end{equation*}
  then $t \cdot v(s) \in U_x$.

  We define a one-parameter family of paths in $N$ by
  \begin{equation*}
    c_s (t) := \exp(t \cdot v(s)), \quad (s,t) \in U_{\epsilon,\delta}
  \end{equation*}
  Note that for each fixed $s$, the path $t \mapsto c_s(t)$ is a
  geodesic with
  \begin{equation}\label{eq:23}
    \| \partial_t c_s(t) \|_\gamma \equiv \| \partial_t c_s(0)
    \|_\gamma = \| v(s) \|_\gamma = 1.
  \end{equation}
  Note also that the image of the family of paths $c_\cdot (\cdot)$ is
  a singular surface in $N$ parametrized by the coordinates $(s,t)$.

  Keeping this in mind, we compute
  \begin{equation} \label{eq:24}
    \begin{aligned}
      \partial_t \gamma \left(
        \partial_s c_s(t), \partial_t c_s(t) \right) &= \gamma \left(
        \frac{\nabla}{\partial t} \partial_s c_s(t), \partial_t c_s(t)
      \right) + \gamma \left(
        \partial_s c_s(t), \frac{\nabla}{\partial t} \partial_t c_s(t)
      \right) \\
      &= \gamma \left( \frac{\nabla}{\partial s} \partial_t
        c_s(t), \partial_t c_s(t)
      \right) \\
      &= \frac{1}{2} \partial_s \gamma \left(
        \partial_t c_s(t), \partial_t c_s(t)
      \right) \\
      &= 0.
    \end{aligned}
  \end{equation}
  
  From (\ref{eq:24}), we immediately see that $\gamma \left(
    \partial_s c_s(t), \partial_t c_s(t) \right)$ is independent of
  $t$.  However, we also have that $c_s(0) = x$ for all $s$, implying
  that $\partial_s c_s(0) = 0$, thus
  \begin{equation*}
    0 = \gamma(\partial_s c_s(0), \partial_t c_s(0)) = \gamma(\partial_s
    c_s(t), \partial_t c_s(t))
  \end{equation*}
  for all $t$.  That is, $\partial_s c_s(t)$ and $\partial_t c_s(t)$
  are orthogonal for all $s$ and $t$.

  We now estimate:
  \begin{align*}
    \| \alpha'(s) \|_\gamma^2 &= \left\| \frac{d}{ds} c_s(r(s)) \right\|_\gamma^2 \\
    &= \left\| \partial_s c_s(r(s)) + r'(s) \partial_r c_s(r(s))
    \right\|_\gamma^2 \\
    &= \left\|
      \partial_s c_s(r(s)) \right\|_\gamma^2 + |r'(s)|^2 \left\|
      \partial_r c_s(r(s))
    \right\|_\gamma^2 \\
    &\geq |r'(s)|^2.
  \end{align*}
  Here, in the third line, we have used orthogonality of $\partial_s
  c_s(t)$ and $\partial_t c_s(t)$.  In the last line, we have used
  (\ref{eq:23}).  Note that equality holds if and only if
  $\left\| \partial_s c_s(r(s)) \right\|_\gamma \equiv 0$.

  Finally, we see that
  \begin{equation*}
    L(\alpha) = \int_0^1 \| \alpha'(s) \|_\gamma \, ds \geq \int_0^1
    |r'(s)| \, ds
    \geq
    \left|
      \int_0^1 r'(s) \, ds
    \right| = |r(1) - r(0)|,
  \end{equation*}
  which proves the desired inequality.  We note that the first
  inequality is an equality if and only if $\left\| \partial_s
    c_s(r(s)) \right\|_\gamma \equiv 0$ (see the previous paragraph) and the
  second inequality is an equality if and only if $r'(s) \geq 0$ for
  all $s$.
\end{proof}

\begin{prop}\label{prop:2}
  Suppose $y \in V_x$ with $\exp_x^{-1}(y) = v$.  Then the path
  \begin{equation*}
    \alpha : [0,1] \rightarrow V_x, \quad \alpha(t) = \exp_x (t \cdot v)
  \end{equation*}
  is of minimal length among all paths in $V_x$ from $x$ to $y$.
  Furthermore, $\alpha$ is the unique minimal path (up to
  reparametrization) in $V_x$ from $x$ to $y$.
\end{prop}

\begin{rmk}
  Note that we will only show that $\alpha$ is minimal only among
  paths (or geodesics) in $V_x$, not all paths (or geodesics) in $N$.
  In particular, we cannot conclude from Proposition \ref{prop:2} that
  $d_\gamma (x,y) = L(\alpha)$, as is done in the case of a strong
  Riemannian manifold, where the neighborhood $V_x$ contains a
  $d_\gamma$-ball of positive radius.
\end{rmk}

\begin{proof}
  A path $\eta(s)$, $s \in [0,1]$, in $V_x$ from $x$ to $y$
  corresponds via $\exp_x^{-1}$ to a path $r(s) \cdot v(s)$ in $U_x$
  with $v(s) \in S_x N$, $r(0) = 0$ and $r(1) \cdot v(1) = v$,
  implying $|r(1)| = \| v \|_\gamma$. By Proposition \ref{prop:4}, we
  therefore have that
  \begin{equation}\label{eq:25}
    L(\eta) \geq \| v \|_\gamma = L(\alpha),
  \end{equation}
  immediately implying minimality of $\alpha$.

  Let equality hold in (\ref{eq:25}).  Again by Proposition
  \ref{prop:4}, this implies $v(s)$ is constant and $r'(s) \geq 0$ for
  all $s$.  However, this means that $\eta$ is just a
  reparametrization of $\alpha$, proving the second statement.
\end{proof}

As an obvious result of Proposition \ref{prop:2}, we get the following
criterion for a weak Riemannian manifold to be a metric space.  It
requires rather strong assumptions which could probably be weakened
significantly, but it will be sufficient for some purposes that we
have in mind.

\begin{thm}\label{thm:4}
  Let $(N,\gamma)$ be a weak Riemannian manifold.  Suppose that for
  each $x \in N$, the exponential mapping $\exp_x$ is a diffeomorphism
  between an open (in the manifold topology) neighborhood $U_x$ of
  $0 \in T_x N$ and $N$.
  
  Then $(N,d_\gamma)$, where $d_\gamma$ is the Riemannian distance function of
  $\gamma$, is a metric space.
\end{thm}

By Propositions \ref{prop:5} and \ref{prop:1}, we see that $\V$ and
$\Mmu$ satisfy the prerequisites of the theorem, so we get the
following corollary.

\begin{cor}
  The weak Riemannian metrics $(\cdot, \cdot)$ and $(\!(\cdot,
  \cdot)\!)$ (cf.~(\ref{eq:1}) and (\ref{eq:2}))
  induce metric space structures on $\Mmu$ and $\V$, respectively.
\end{cor}

$(\M, (\cdot, \cdot))$ itself certainly does not satisfy the
prerequisites of Theorem \ref{thm:4}, as is explicitly shown in
\cite[\S 3]{gil-medrano91:_rieman_manif_of_all_rieman_metric}.
Therefore, we will have to use a different strategy in this case.

\section{$\M$ is a metric space}\label{sec:basic-metr-geom}

Proving that $d$ is a metric will be done by finding a manifestly
positive-definite metric (in the sense of metric spaces) on $\M$ that
in some way bounds the $d$-distance between two points from below,
implying that it is positive.  First, though, we will prove a
preliminary result that bounds from below the distance between metrics
with inducing differing volume forms.

\subsection{Lipschitz continuity of the square root of the volume}\label{sec:lipsch-cont-square}

Using the following lemma as a first step to proving that $d$ is a
metric takes its inspiration from \cite[\S
3.3]{michor05:_vanis_geodes_distan_spaces_of}.

\begin{lem}\label{lem:13}
  Let $g_0, g_1 \in \M$.  Then for any measurable subset $Y \subseteq
  M$,
  \begin{equation*}
    \left| \sqrt{\Vol(Y,g_1)} - \sqrt{\Vol(Y,g_0)} \right| \leq \frac{\sqrt{n}}{4} d(g_0,g_1).
  \end{equation*}
\end{lem}
\begin{proof}
  Let $g_t$, $t \in [0,1]$, be any path from $g_0$ to $g_1$, and
  define $h_t := g'_t$.  We compute
  \begin{equation}\label{eq:26}
    \begin{aligned}
      \partial_t \Vol(Y, g_t) &= \partial_t \int_Y \, \mu_{g_t} =
      \int_Y \partial_t \, \mu_{g_t} = \int_Y \frac{1}{2} \tr_{g_t}
      (h_t)
      \, \mu_{g_t} \\
      &\leq \left( \int_Y \, \mu_{g_t} \right)^{1/2} \left(
        \frac{1}{4}
        \int_Y \tr_{g_t}(h_t)^2 \, \mu_{g_t} \right)^{1/2} \\
      &\leq \frac{1}{2} \sqrt{\Vol(Y,g_t)} \left( \int_M
        \tr_{g_t}(h_t)^2 \, \mu_{g_t} \right)^{1/2},
    \end{aligned}
  \end{equation}
  where the first line follows from
  \cite[Lemma~2.38]{clarked):_compl_of_manif_of_rieman}, the second
  line follows from Hölder's inequality, and the last line from the
  nonnegativity of $\tr_{g_t}(h_t)^2$.

  It is not hard to see $(0,2)$-tensors $k$ with $\tr_{g_t} k = 0$ are
  orthogonal to those of the form $\rho \cdot g_t$ with $\rho \in
  C^\infty(M)$.  Therefore, letting $h_t^T$ be the traceless part of
  $h_t$, we have
  \begin{equation*}
    \tr_{g_t}(h_t^2) = \tr_{g_t}\left((h^T_t)^2\right) + \frac{1}{n} \tr_{g_t}(h_t)^2.
  \end{equation*}
  This implies
  \begin{equation*}
    \tr_{g_t}(h_t)^2 = n \left( \tr_{g_t}(h_t^2) -
      \tr_{g_t}\left((h^T_t)^2\right) \right) \leq n \tr_{g_t}(h_t^2),
  \end{equation*}
  since $\tr_{g_t}\left((h^T_t)^2\right) \geq 0$.  Applying this to
  (\ref{eq:26}) gives
  \begin{equation}\label{eq:27}
    \begin{aligned}
      \partial_t \Vol(Y, g_t) &\leq \frac{1}{2} \sqrt{\Vol(Y,g_t)}
      \left( n \int_M \tr_{g_t}(h_t^2) \, \mu_{g_t} \right)^{1/2} \\
      &\leq \frac{\sqrt{n}}{2} \sqrt{\Vol(Y,g_t)} \| h_t \|_{g_t}.
    \end{aligned}
  \end{equation}

  We next compute
  \begin{equation}\label{eq:49}
    \begin{aligned}
      \sqrt{\Vol(Y,g_1)} - \sqrt{\Vol(Y,g_0)} &= \int_0^1 \partial_t
      \sqrt{\Vol(Y,g_t)} \, dt = \int_0^1 \frac{1}{2} \frac{\partial_t
        \Vol(Y,g_t)}{\sqrt{\Vol(Y,g_t)}} \, dt \\
      &\leq \int_0^1 \frac{\sqrt{n}}{4} \| h_t \|_{g_t} \, dt 
      = \frac{\sqrt{n}}{4} L(g_t),
    \end{aligned}
  \end{equation}
  where the inequality follows from (\ref{eq:27}).  Since this holds
  for all paths from $g_0$ to $g_1$, and we can repeat the computation
  with $g_0$ and $g_1$ interchanged, it implies the result
  immediately.
\end{proof}

We note that Lemma \ref{lem:13} in particular gives a positive lower
bound on the distance between two metrics in $\M$ that have different
total volumes---so we must now deal with the case where the two
metrics have the same total volume.

\subsection{A (positive-definite) metric on
  $\M$}\label{sec:another-metric-m}

For the remainder of the paper, we fix an arbitrary reference metric
$g \in \M$.

We begin by defining a function on $\M \times \M$ and showing that it is
indeed a metric.

\begin{dfn}\label{dfn:15}
  Consider $\M_x = \{ \tilde{g} \in \satx \mid \tilde{g} > 0 \}$
  (cf.~Section \ref{sec:manif-riem-metr}).  Define a Riemannian
  metric $\langle \cdot , \cdot \rangle^0$ on $\M_x$ given by
  \begin{equation*}
    \langle h , k \rangle^0_{\tilde{g}} = \tr_{\tilde{g}} (h k) \det
    g(x)^{-1} \tilde{g} \quad \forall h, k \in T_{\tilde{g}} \M_x \cong
    \satx.
  \end{equation*}
  (Recall that $g \in \M$ is our fixed reference element.)  We denote
  by $\theta^g_x$ the Riemannian distance function of $\langle \cdot ,
  \cdot \rangle^0$.
\end{dfn}

Note that $\theta^g_x$ is automatically positive definite, since it is
the distance function of a Riemannian metric on a finite-dimensional
manifold.  By integrating it in $x$, we can pass from a metric on
$\M_x$ to a function on $\M \times \M$ as follows:

\begin{dfn}\label{dfn:16}
  For any measurable $Y \subseteq M$, define a function $\Theta_Y : \M
  \times \M \rightarrow \R$ by
  \begin{equation*}
    \Theta_Y(g_0, g_1) = \integral{Y}{}{\theta^g_x(g_0(x), g_1(x))}{\mu_g(x)}.
  \end{equation*}
\end{dfn}

We have omitted the metric $g$ from the notation for $\Theta_Y$.  The
next lemma justifies this choice.

\begin{lem}\label{lem:28}
  $\Theta_Y$ does not depend on the choice of $g \in \M$ in the above
  definition.  That is, if we choose any other $\tilde{g} \in \M$ and
  define $\langle \cdot , \cdot \rangle^0$ and $\theta^{\tilde{g}}_x$
  with respect to this new reference metric, then
  \begin{equation*}
    \integral{Y}{}{\theta^g_x(g_0(x), g_1(x))}{\mu_g(x)} =
    \integral{Y}{}{\theta^{\tilde{g}}_x(g_0(x), g_1(x))}{\mu_{\tilde{g}}(x)}
  \end{equation*}
\end{lem}
\begin{proof}
  Let $\tilde{g} \in \M$ be any other metric.  Recall that
  $\theta^g_x$ was the distance function associated to the Riemannian
  metric $\langle \cdot, \cdot \rangle^0$ on $\Matx$, and the metric
  $g$ enters in the definition of this Riemannian metric.  Take a path
  $g_t(x)$ in $\Matx$.  For now, let's put $g$ and $\tilde{g}$ back in
  the notation, so that we can write formulas unambiguously.

  Using the definitions of $\Theta^g_Y$ and $\theta^g_x$, where infima
  are always taken over paths $g_t(x)$ from $g_0(x)$ to $g_1(x)$, and
  where $h_t(x) := g_t(x)'$, we
  can compute:
  \begin{align*}
    \Theta^g_Y(g_0,g_1) &= \int_Y \theta^g_x(g_0(x), g_1(x)) \, \mu_g(x) \\
    &= \int_Y \left( \inf \int_0^1 \sqrt{\tr_{g_t(x)}(h_t(x)^2)
        \frac{\det g_t(x)}{\det g(x)}} \, dt \right) \sqrt{\det g(x)} \, dx^1 \cdots
    dx^n \\
    &= \int_Y \left( \inf \int_0^1 \sqrt{\tr_{g_t(x)}(h_t(x)^2)
        \frac{\det g_t(x)}{\det \tilde{g}(x)}} \, dt \right) \sqrt{\det \tilde{g}(x)} \, dx^1 \cdots
    dx^n \\
    &= \Theta^{\tilde{g}}_Y(g_0,g_1).
  \end{align*}
\end{proof}

\begin{lem}\label{lem:44}
  Let any $Y \subseteq M$ be given.  Then $\Theta_Y$ is a pseudometric
  on $\M$, and $\Theta_M$ is a metric (in the sense of metric spaces).
  
  Furthermore, if $Y_1 \subset Y_2$, then $\Theta_{Y_1}(g_0, g_1)
  \leq \Theta_{Y_2}(g_0, g_1)$ for all $g_0, g_1 \in \M$.
\end{lem}
\begin{proof}
  Nonnegativity, vanishing distance for equal elements, symmetry and
  the triangle inequality are clear from the corresponding properties
  for $\theta^g_x$.

  That $\Theta_M$ is positive definite is also not hard to prove.
  Since $\theta^g_x$ is a metric on $\M_x$, $\theta^g_x(g_0(x),
  g_1(x))$ is positive whenever $g_0(x) \neq g_1(x)$.  But since $g_0$ and
  $g_1$ are smooth metrics, if they differ at a point, they differ
  over an open neighborhood of that point.  Hence the integral of
  $\theta^g_x(g_0(x), g_1(x))$ must be positive.

  The second statement follows immediately from nonnegativity of
  $\theta^g_x$.
\end{proof}

\subsection{Proof of the main result}\label{sec:proof-main-result}

We have set up almost everything we need to prove the main result of
this section.  The last preparation comes down to using $\Theta_Y$ to
provide a lower bound for the distance between elements of $\M$ as
measured by $d$.

\begin{prop}\label{prop:20}
  For any $Y \subseteq M$ and $g_0, g_1 \in \M$, we have the following
  inequality:
  \begin{equation*}
    \Theta_Y(g_0, g_1) \leq d(g_0, g_1) \left( \sqrt{n}\, d(g_0, g_1) +
      2 \sqrt{\Vol(M, g_0)} \right).
  \end{equation*}
  In particular, $\Theta_Y$ is a continuous pseudometric (w.r.t.~$d$).
\end{prop}
\begin{proof}
  By Lemma \ref{lem:44}, we need only prove the inequality for $Y =
  M$, and then it follows for any subset.

  We can clearly find a path $g_t$ from $g_0$ to $g_1$ with $L(g_t)
  \leq 2 d(g_0, g_1)$.  Then for any $\tau \in [0,1]$, we get
  \begin{equation*}
    2 d(g_0, g_1) \geq L(g_t) \geq L \left( g_t|_{[0,\tau]} \right) \geq d(g_0,
    g_\tau) \geq \frac{4}{\sqrt{n}} \left| \sqrt{\Vol(M, g_\tau)} -
      \sqrt{\Vol(M, g_0)} \right|,
  \end{equation*}
  where the last inequality is Lemma \ref{lem:13}.  In particular, we
  get
  \begin{equation}\label{eq:48}
    \sqrt{\Vol(M, g_\tau)} \leq \sqrt{\Vol(M, g_0)} +
    \frac{\sqrt{n}}{2} d(g_0, g_1) =: V
  \end{equation}
  for all $\tau \in [0,1]$.

  To find the length of $g_t$, we first integrate $\langle g'_t, g'_t
  \rangle$ over $x \in M$, then take the square root, and finally
  integrate over $t$.  Ideally, we would wish to change the order of
  integration, so that we first integrate over $t$, then over $x$.  We
  cannot do this exactly, but we can bound the computation of the
  length from below by an expression where we integrate in the
  opposite order, and this expression will involve $\theta^g_x$ and
  $\Theta_M$.  So let's see how this works.

  Let $h_t := g'_t$.  From Hölder's inequality,
  \begin{equation*}
    \int_M \sqrt{\tr_{g_t} (h_t^2)} \, d \mu_{g_t} \leq \left( \int_M
      \, d \mu_{g_t} \right)^{1/2} \left( \int_M \tr_{g_t} (h_t^2)
      \, d \mu_{g_t} \right)^{1/2},
  \end{equation*}
  which gives
  \begin{equation}\label{eq:30}
    \begin{aligned}
      \| h_t \|_{g_t} &= \left( \int_M \tr_{g_t} (h_t^2) \, d
        \mu_{g_t} \right)^{1/2} \geq \frac{1}{\sqrt{\Vol(M,g_t)}}
      \int_M
      \sqrt{\tr_{g_t} (h_t^2)} \, d \mu_{g_t} \\
      &\geq \frac{1}{V} \int_M \sqrt{\tr_{g_t} (h_t^2)} \, d
      \mu_{g_t},
    \end{aligned}
  \end{equation}
  where we have also used \eqref{eq:48}.  To remove the
  $t$-dependence from the volume element, we use
  \begin{equation*}
    \mu_{g_t} = \frac{\sqrt{\det g_t}}{\sqrt{\det g}} \mu_{g} =
    \sqrt{\det g^{-1} g_t} \mu_{g}.
  \end{equation*}
  We then rewrite (\ref{eq:30}) as
  \begin{equation}\label{eq:31}
    \| h_t \|_{g_t} \geq \frac{1}{V} \int_M
    \sqrt{\tr_{g_t}(h_t^2) \det g^{-1} g_t} \, \mu_{g} = \frac{1}{V}
    \integral{M}{}{\sqrt{\langle h_t(x) , h_t(x) \rangle^0_{g_t(x)}}}{\mu_g(x)},
  \end{equation}
  where we have used the Riemannian metric $\langle \cdot , \cdot
  \rangle^0$ on $\M_x$ (cf.~Definition \ref{dfn:15}).

  Since we have removed the $t$-dependence from the measure above, we
  can change the order of integration in the calculation of the length
  of $g_t$:
  \begin{equation*}
    \begin{aligned}
      L(g_t) &= \integral{0}{1}{\| h_t \|_{g_t}}{d t} \geq \frac{1}{V}
      \integral{0}{1}{\integral{M}{}{\sqrt{\langle h_t(x),
            h_t(x) \rangle^0_{g_t(x)}}}{\mu_g(x)}}{d t} \\
      &= \frac{1}{V} \integral{M}{}{\integral{0}{1}{\sqrt{\langle
            h_t(x), h_t(x) \rangle^0_{g_t(x)}}}{d t}}{\mu_g(x)}.
    \end{aligned}
  \end{equation*}
  Now we concentrate on the $t$-integral in the expression above.
  Since $g_t(x)$ is a path in $\M_x$ from $g_0(x)$ to $g_1(x)$ with
  tangents $h_t(x)$, the $t$-integral is actually the length of
  $g_t(x)$ with respect to $\langle \cdot , \cdot \rangle^0$.  But by
  definition, this length is bounded from below by $\theta^g_x(g_0(x),
  g_1(x))$.  Therefore, we get the estimate
  \begin{equation*}
    L(g_t) \geq \frac{1}{V} \integral{M}{}{\theta^g_x(g_0(x),
      g_1(x))}{\mu_g(x)} = \frac{1}{V} \Theta_M(g_0, g_1).
  \end{equation*}
  But now the result is immediate given \eqref{eq:48} and the fact
  that we have assumed $L(g_t) \leq 2 d(g_0, g_1)$.
\end{proof}

Because $\Theta_M$ is positive-definite, the previous proposition
immediately implies that the pseudo-metric $d$ is as well.  Thus, we
have the main result of the paper.

\begin{thm}\label{thm:6}
  $(\M, d)$, where $d$ is the distance function induced from the $L^2$
  metric $(\cdot, \cdot)$, is a metric space.
\end{thm}

\begin{rmk}\label{rmk:2}
  As already noted, in
  \cite{michor06:_rieman_geomet_spaces_of_plane_curves}, Michor and
  Mumford studied weak Riemannian manifolds where the induced distance
  between all points is zero.  They also showed that the sectional
  curvature is unbounded from above at each point---thus, the
  vanishing of the distance function can be related to the fact that
  the space curls up on itself arbitrarily tightly.  In a sense, our
  theorem confirms this ``mechanism'' for allowing the induced
  distance to vanish, since as we already noted in Section
  \ref{sec:manif-riem-metr}, the sectional curvature of $(\M, (\cdot,
  \cdot))$ is bounded from above---indeed, it is
  nonpositive---preventing the pathologies found by Michor and
  Mumford.
\end{rmk}

Theorem \ref{thm:6} is a first step in investigating the metric
geometry of $\M$---that is, it is certainly a prerequisite for $\M$
itself (rather than the quotient space where elements at distance zero
are identified) to have any metric geometry at all.  In a forthcoming
paper \cite{clarke:_compl_of_manif_of_rieman_metric}, motivated by the
appearance of $(\M, (\cdot, \cdot))$ in Teichmüller theory and the
important results in that field concerning the completion of the
Weil-Petersson metric, we will continue this study by giving a
description of the completion of $(\M, d)$.

\bibliography{main} \bibliographystyle{hamsplain}

\end{document}